\documentclass[11pt]{amsart}
\usepackage[margin=1in]{geometry}

\usepackage{amssymb}
\usepackage{amsthm}
\usepackage{amsmath}
\usepackage{mathrsfs}
\usepackage{amsbsy}
\usepackage[all]{xy}
\usepackage{bm}
\usepackage{hyperref}
\usepackage{tikz}
\usepackage{array}
\usepackage{enumerate}
\usepackage{xcolor}
\usepackage{bbm}
\usepackage{comment}
\usepackage{dynkin-diagrams}
\usepackage{ifthen}
\usepackage{thmtools, thm-restate}

\usepackage[noabbrev,capitalise,nameinlink]{cleveref}

\definecolor{citations}{rgb}{0,0.75,0}

\hypersetup{colorlinks=true, citecolor=citations, linkcolor=citations,urlcolor=citations}

\usepackage{hhline}
\allowdisplaybreaks
\usepackage[noadjust]{cite}
\usepackage{asymptote}

\usepackage{caption}
\usepackage{tabu}
\usepackage{diagbox}
\usepackage[noabbrev,capitalise,nameinlink]{cleveref}
\crefname{conjecture}{Conjecture}{Conjectures}

\newtheorem{theorem}{Theorem}[section]
\newtheorem{proposition}[theorem]{Proposition}

\newtheorem{lemma}[theorem]{Lemma}

\theoremstyle{definition}
\newtheorem{definition}[theorem]{Definition}

\newtheorem{example}[theorem]{Example}

\newcommand{\rk}{\mathrm{rk}}

\newcommand{\TT}{\mathsf{T}}
\newcommand{\CC}{\mathsf{C}}
\newcommand{\PT}{\mathrm{PT}}
\newcommand{\Orn}{\mathrm{Orn}}
\newcommand{\OO}{\mathcal{O}}
\newcommand{\Pop}{\mathsf{Pop}} 
\newcommand{\Orb}{\mathrm{Orb}} 
\newcommand{\height}{\mathrm{height}} 
\newcommand{\depth}{\mathrm{depth}} 
\newcommand{\MM}{\mathcal{M}}

\newcommand{\de}{\delta}
\newcommand{\ch}{\mathrm{ch}}
\newcommand{\rt}{\mathfrak{r}} 
\newcommand{\WT}{\textsc{wt}}

\newcommand{\dfn}[1]{\textcolor{red}{\emph{#1}}}

\begin{document}

\title[]{The Pop-Stack Operator on Ornamentation Lattices}
\subjclass[2010]{}

\author[]{Khalid Ajran}
\address[]{Department of Mathematics, Massachusetts Institute of Technology, Cambridge, MA 02138, USA}
\email{kajran@mit.edu}

\author[]{Colin Defant}
\address[]{Department of Mathematics, Harvard University, Cambridge, MA 02138, USA}
\email{colindefant@gmail.com}

\maketitle

\begin{abstract}
Each rooted plane tree $\mathsf{T}$ has an associated \emph{ornamentation lattice} $\mathcal{O}(\mathsf{T})$. The ornamentation lattice of an $n$-element chain is the $n$-th Tamari lattice. We study the pop-stack operator $\mathsf{Pop}\colon\mathcal{O}(\mathsf{T})\to\mathcal{O}(\mathsf{T})$, which sends each element $\delta$ to the meet of the elements covered by or equal to $\delta$. We compute the maximum size of a forward orbit of $\Pop$ on $\mathcal{O}(\mathsf{T})$, generalizing a result of Defant for Tamari lattices. We also characterize the image of $\Pop$ on $\mathcal{O}(\mathsf{T})$, generalizing a result of Hong for Tamari lattices. For each integer $k\geq 0$, we provide necessary conditions for an element of $\mathcal{O}(\mathsf{T})$ to be in the image of $\Pop^k$. This allows us to completely characterize the image of $\Pop^k$ on a Tamari lattice.   
\end{abstract} 

\section{Introduction}\label{sec:intro}
\subsection{Ornamentation Lattices} 
Laplante-Anfossi \cite{Laplante} recently introduced remarkable polytopes called \emph{operahedra}. These polytopes are special examples of \emph{graph associahedra} \cite{CarrDevadoss,Postnikov} and \emph{poset associahedra} \cite{Galashin}, and they generalize both the associahedra and permutohedra. Laplante-Anfossi also considered certain posets obtained by orienting the $1$-skeletons of operahedra, and he asked if these posets are lattices. Defant and Sack \cite{DefantSack} answered Laplante-Anfossi's question in the affirmative, naming his posets \emph{operahedron lattices}. A key idea that Defant and Sack employed was to embed an operahedron lattice into the product of an interval in the weak order on the symmetric group and another lattice. This other lattice, which is the primary focus of this article, is a new structure called an \emph{ornamentation lattice}. 

Let $\PT_n$ denote the set of rooted plane trees with $n$ nodes. Let $\TT\in\PT_n$. We view $\TT$ both as a graph and as a poset in which every non-root element is covered by exactly $1$ element (so the root is the unique maximal element). We denote the partial order on $\TT$ by $\leq_\TT$. For $v\in\TT$, we write 
\[\Delta_{\TT}(v)=\{v'\in\TT: v'\leq_\TT v\}\quad\text{and}\quad\nabla_\TT(v)=\{v'\in\TT:v\leq_\TT v'\}.\] An \dfn{ornament} of $\TT$ is a set of nodes that induces a connected subgraph of $\TT$. Let $\Orn(\TT)$ denote the set of ornaments of $\TT$. We will often view an ornament as a rooted plane tree. Two sets are said to be \dfn{nested} if one is a subset of the other. An \dfn{ornamentation} of $\TT$ is a function $\delta\colon\TT\to\Orn(\TT)$ such that 
\begin{itemize}
\item for every $v\in\TT$, the unique maximal element of $\delta(v)$ is $v$; 
\item for all $v,v'\in\TT$, the ornaments $\delta(v)$ and $\delta(v')$ are either nested or disjoint. 
\end{itemize}
Let $\OO(\TT)$ denote the set of ornamentations of $\TT$. We view $\OO(\TT)$ as a poset, where the partial order $\leq$ is defined so that $\delta\leq\delta'$ if and only if $\delta(v)\subseteq\delta'(v)$ for all $v\in\TT$. Then $\OO(\TT)$ is a lattice whose meet operation $\wedge$ is such that \[(\delta\wedge\delta')(v)=\delta(v)\cap\delta'(v)\] for all $v\in\TT$. The minimum element $\delta_{\min}$ and maximum element $\delta_{\max}$ of $\OO(\TT)$ satisfy \[\delta_{\min}(v)=\{v\}\quad\text{and}\quad\delta_{\max}(v)=\Delta_\TT(v)\] for all $v\in\TT$. See \cref{fig:big_lattice}. If $\de\in\OO(\TT)$ and $v\in\TT$, then we call $\de(v)$ an ornament of $\de$. 

Ornamentation lattices naturally generalize Tamari lattices. Indeed, the $n$-th Tamari lattice is isomorphic to the ornamentation lattice of an $n$-element chain. Tamari lattices have numerous interesting properties; we will generalize some of these properties to the much broader family of ornamentation lattices.  

\begin{figure}[ht]
\begin{center}{\includegraphics[height=13cm]{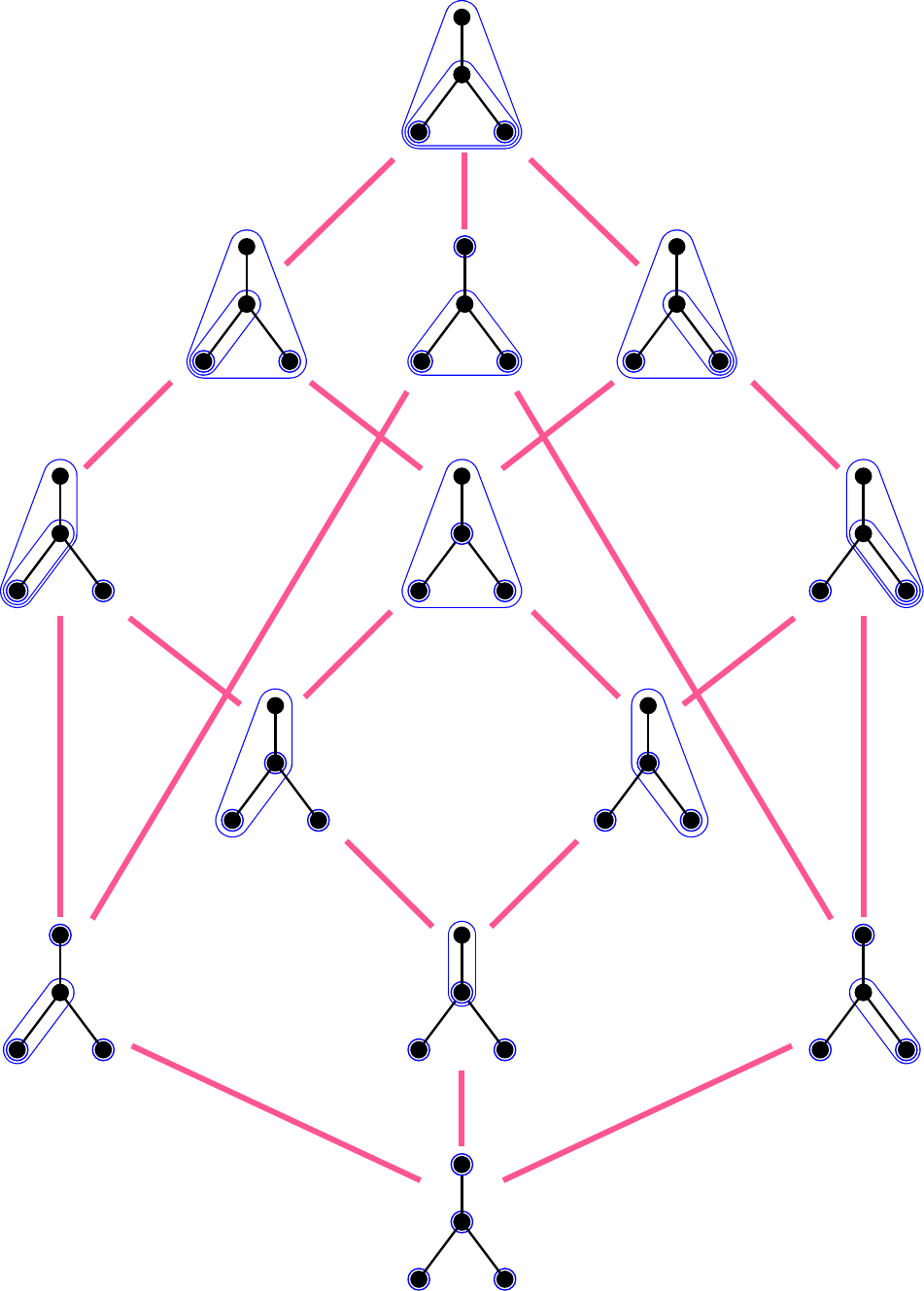}}
\end{center}
\caption{The ornamentation lattice of a rooted tree with $4$ nodes. 
} 
\label{fig:big_lattice}
\end{figure} 

\subsection{Pop-Stack Operators}
The \dfn{pop-stack operator} of a finite lattice $L$ is the map $\Pop\colon L\to L$ defined by 
\[\Pop(x)=\bigwedge(\{x\}\cup\{y\in L:y\lessdot x\}).\] Initial works on this operator focused on the case where $L$ is the weak order on the symmetric group \cite{Asinowski, Asinowski2, ClaessonPop, ClaessonPantone, Lichev, PudwellSmith, Ungar}, where it can be interpreted as a deterministic sorting procedure that makes use of a data structure called a \emph{pop-stack}. Several recent articles have studied pop-stack operators on other interesting lattices, especially from a dynamical point of view \cite{BarnardPop, ChoiSun, DefantCoxeterPop, DefantMeeting, Semidistrim, Hong}. For example, Defant \cite{DefantMeeting} and Hong \cite{Hong} found that the pop-stack operator on a Tamari lattice has particularly well behaved dynamics. Choi and Sun \cite{ChoiSun} and Barnard, Defant, and Hanson \cite{BarnardPop} later generalized these results to the much broader class of \emph{Cambrian lattices}. In this article, we generalize the results about Tamari lattices in a different direction by considering pop-stack operators of ornamentation lattices.  

Given a finite lattice $L$ and an element $x\in L$, we let $\Orb_{\Pop}(x)=\{x,\Pop(x),\Pop^2(x),\ldots\}$ denote the \dfn{forward orbit} of $x$ under $\Pop$. If $t$ is a sufficiently large integer, then $\Pop^t(x)$ is the minimum element $\hat{0}$ of $L$. We are interested in $|\Orb_{\Pop}(x)|$. When considering the pop-stack operator, one of the most well studied quantities associated to $L$ is 
\[\max_{x\in L}|\Orb_{\Pop}(x)|\] (see, e.g., \cite{BarnardPop, DefantCoxeterPop, DefantMeeting, Ungar}). For example, Barnard, Defant, and Hanson found that if $L$ is a Cambrian lattice of a finite Coxeter group $W$, then $\max_{x\in L}|\Orb_{\Pop}(x)|$ is equal to the Coxeter number of $W$. In particular, if $L$ is the $n$-th Tamari lattice (which is a Cambrian lattice of the symmetric group $S_n$), then this number is $n$; this special case was proved earlier in \cite{DefantMeeting}. 

Another important aspect of the pop-stack operator is its image, which has been studied for a variety of lattices in \cite{BarnardPop, ChoiSun, Semidistrim, Hong, Sapounakis}. Defant and Williams found numerous different interpretations of the size of the image of $\Pop$ when $L$ is a semidistributive lattice (or, more generally, a \emph{semidistrim} lattice). Notably, when $L$ is semidistributive, the image of $\Pop$ is in bijection with the facets of a simplicial complex called the \emph{canonical join complex} of $L$ (see \cite{Barnard,BarnardPop}). Hong \cite{Hong} characterized and enumerated $\Pop$ images on Tamari lattices. Choi and Sun \cite{ChoiSun} subsequently obtained analogous results for a variety of other lattices. Barnard, Defant, and Hanson then unified several of these results by providing a simple Coxeter-theoretic description of the image of $\Pop$ on an arbitrary Cambrian lattice of a finite Coxeter group. We aim to generalize Hong's results concerning Tamari lattices in a different direction by studying $\Pop$ images on ornamentation lattices. 

\subsection{Main Results} 

Let $\TT\in\PT_n$. Unless otherwise stated, we will write $\Pop$ for the pop-stack operator on the ornamentation lattice $\OO(\TT)$. Define the \dfn{depth} of a node $v\in\TT$ to be the quantity $\depth_\TT(v)=|\nabla_\TT(v)|-1$. Let $\MM_\TT$ denote the set of maximal chains in $\TT$. 

\begin{theorem}\label{thm:forward_orbit}
Let $\TT\in\PT_n$, and let $\rt$ be the root of $\TT$. Then 
\[\max_{\delta\in\OO(\TT)}|\Orb_\Pop(\delta)|=\max_{C\in\MM_{\TT}}\min_{v\in C\setminus\{{\rt}\}}(|\Delta_\TT(v)|+2\,\depth_\TT(v)-1).\]
\end{theorem} 

Let $\CC_n$ denote the chain with $n$ nodes. The ornamentation lattice $\OO(\CC_n)$ is the $n$-th Tamari lattice. In this special case, \cref{thm:forward_orbit} tells us that the maximum size of a forward orbit of $\Pop$ on the $n$-th Tamari lattice is $n$, which agrees with the aforementioned result from \cite{DefantMeeting}.

The next theorem characterizes the image of the pop-stack operator on an ornamentation lattice, thereby generalizing a result of Hong from \cite{Hong}. 

\begin{theorem}\label{thm:pop_image_intro} 
Let $\TT\in\PT_n$ and $\de\in\OO(\TT)$. For each $u\in\TT$, let $o_u$ be the minimal ornament of $\de$ that properly contains $\de(u)$; if no such ornament exists, let $o_u=\TT$. Then $\de\in\Pop(\OO(\TT))$ if and only if for every $u\in\TT$ and every child $u'$ of $u$ in $\de(u)$, we have $\Delta_{\de(u)}(u')\neq\Delta_{o_u}(u')$. 
\end{theorem}

In \cref{sec:semidistributivity}, we will prove that every ornamentation lattice $\OO(\TT)$ is in fact semdistributive. This means that $\OO(\TT)$ has a canonical join complex, the number of facets of which is $|\Pop(\OO(\TT))|$. This provides an alternative interpretation of \cref{thm:pop_image_intro}. See \cite[Corollary~9.10]{Semidistrim} for other interpretations of the size of the image of $\Pop$ on a semidistributive lattice. 

It is also natural to consider the image of $\Pop^k$ when $k\geq 2$. So far, this has not been explored in the literature, largely because it is often difficult to characterize this image. We will prove some necessary conditions for an ornamentation to be in the image of $\Pop^k$ (see \cref{lem:vrank,cla:popk_beads}). These conditions are, in general, not sufficient. However, we will obtain as a consequence a complete characterization of the image of $\Pop^k(\OO(\CC_n))$, where $\CC_n$ is an $n$-element chain poset (see \cref{thm:popk_tamari}). This will allow us to deduce the following explicit generating function for the sizes of the images of $\Pop^k$ on Tamari lattices.

\begin{theorem}\label{thm:gf}
For each nonnegative integer $k$, we have 
\[\sum_{n\geq 0}|\Pop^k(\OO(\CC_n))|x^n=\frac{1-x^{k+1}-\sqrt{(1-x^{k+1})^2-4x(1-x)(1-x^{k+1})}}{2x(1-x)}.\] 
\end{theorem}

When $k=0$, the generating function in \cref{thm:gf} is that of the Catalan numbers, which we already knew since $\Pop^0(\OO(\CC_n))=\OO(\CC_n)$ is the $n$-th Tamari lattice. When $k=1$, the generating function is that of the Motzkin numbers; this recovers one of Hong's main results from \cite{Hong}. 

\subsection{Outline} 
\cref{sec:definitions} discusses basic properties of the pop-stack operator on an ornamentation lattice and introduces some notation and terminology. We prove \cref{thm:forward_orbit} in \cref{sec:maxorbit} after establishing a necessary condition for an ornamentation to be in the image of $\Pop^k$ (\cref{lem:vrank}). We prove \cref{thm:pop_image_intro} in \cref{sec:image}. In \cref{sec:image_iterate}, we prove \cref{cla:popk_beads}, which yields another constraint on the image of $\Pop^k$. We then characterize the image of $\Pop^k$ on the $n$-th Tamari lattice in \cref{thm:popk_tamari}, and we deduce \cref{thm:gf}. Finally, \cref{sec:semidistributivity} provides a short proof that ornamentation lattices are semidistributive. 

\section{Basics}\label{sec:definitions}

Let $\TT \in \PT_n$ be a rooted plane tree with root node $\mathfrak{r}$. All of our main results are trivial when $n=1$, so we will tacitly assume that $n\geq 2$. The \dfn{length} of a chain with $\ell$ elements is $\ell-1$. For $v \in \TT$, let $\height_\TT(v)$ and $\depth_\TT(v)$ denote the \dfn{height} and \dfn{depth} of $v$ in $\TT$. That is, $\height_\TT(v)$ is the maximum length of a chain in $\Delta_\TT(v)$, and $\depth_\TT(v)$ is the length of the chain $\nabla_\TT(v)$. A \dfn{descendant} of $v$ is an element of $\Delta_\TT(v)\setminus\{v\}$, and a \dfn{child} of $v$ is an element covered by $v$. Let $\ch_\TT(v)$ be the set of children of $v$. Let $\MM_\TT$ be the set of maximal chains of $\TT$. For $C \in \MM_\TT$ and $v\in C$, let $\height_C(v)$ denote the height of $v$ within $C$; if $v$ is not a leaf, then we let $\ch_C(v)$ be the unique child of $v$ in $C$. For a node $v$ with a descendant $u$, let $(u, v)$ denote the set of nodes (excluding $u$ and $v$) in the unique path from $u$ to $v$. We also write $[u, v)=(u,v)\cup\{u\}$, $(u, v]=(u,v)\cup\{v\}$, and $[u, v]=(u,v)\cup\{u,v\}$. 

For $C \in \MM_\TT$ and $v\in C$, let $b_C(v) = |\Delta_\TT(v)| - \height_C(v) - 1$ be the number of descendants of $v$ that do not lie in $C$. Define $f_C\colon C \rightarrow \mathbb{N}_0$ recursively as follows: 
\begin{equation}\label{eq:fC1} f_C(v) = \begin{cases} 
          0 & \text{if $v$ is a leaf;} \\
          
          f_C(\ch_C(v)) + 1 & \text{if }  f_C(\ch_C(v)) + 1 \leq b_C(\ch_C(v)); \\
          
          f_C(\ch_C(v)) & \text{otherwise.}
       \end{cases} 
    \end{equation}
In other words, if we walk up the chain $C$, then the value of $f_C$ starts at $0$, and each time we move to a parent node $v$, the value increases by $1$ so long as this does not cause $f_C(v)$ to exceed $b_C(\ch(v))$. 
Utilizing the identity $\depth(v) = |C| - \height_C(v) - 1$, one can show inductively that 
\begin{equation}\label{eq:fC2} f_C(v) = \begin{cases} 
          0 & \text{if $v$ is a leaf;} \\
          
          \min_{u \in  \Delta_C(v)\setminus\{v\}}(|\Delta_\TT(u)| + 2\depth_\TT(u)) - |C| - \depth_\TT(v) - 1& \text{otherwise.}
       \end{cases} 
    \end{equation}    

For each $v \in \TT$, define $\TT^*_v$ to be the tree obtained from $\TT$ by appending a chain with $f_C(v)$ nodes to the bottom of each chain $C\in\MM_\TT$ containing $v$. The following definition will give us a strong tool for understanding the maximum size of a forward orbit under $\Pop$.  

\begin{definition}\label{vrankdef}
Let $v\in\TT$. The \dfn{$v$-rank} of a node $u \in \TT$ is 
\[ \rk_v(u) = \begin{cases} 
          -1 & \text{if } u \not\in \Delta_\TT(v); \\

          \infty & \text{if } u = v; \\
          
          \height_{\TT^*_v}(u) & \text{otherwise.}
       \end{cases}
    \]
\end{definition}

\begin{figure}[ht]
\begin{center}{\includegraphics[height=10.598cm]{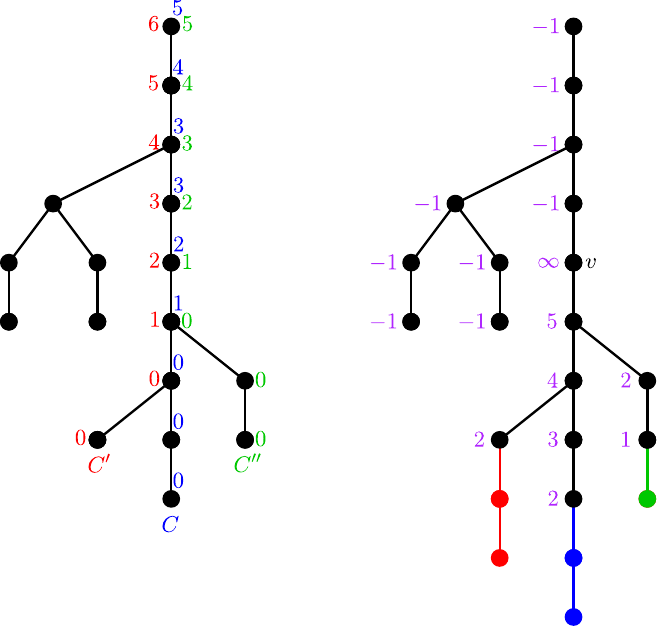}}
\end{center}
\caption{On the left is a plane tree $\TT$. Three of the maximal chains, named $C$, $C'$, and $C''$, are labeled under their leaves. For each $\widetilde{C}\in\{C,C',C''\}$, each node $u\in\widetilde C$ is labeled with the value of $f_{\widetilde{C}}(u)$. On the right is the tree $\TT_v^*$, where $v$ is as indicated. This tree is obtained by adding $f_{\widetilde{C}}(v)$ new nodes to the bottom of each chain $\widetilde{C}\in\{C,C',C''\}$. On the right, each node in $\TT$ is labeled with its $v$-rank. } 
\label{fig:fC}
\end{figure} 

See \cref{fig:fC} for an example of the above definitions. For each maximal chain $C\in\MM_\TT$, the values of $f_C$ are weakly increasing as we move up $C$. It follows that if $v$ is a child of $v'$, then $\TT_v^*$ is a subtree of $\TT_{v'}^*$. This implies that for every $v\in\TT$, the tree $\TT_v^*$ is a subtree of $\TT_{\rt}^*$. Moreover, the maximum of $\rk_{\rt}(u)$ over all non-root nodes $u$ is attained when $u$ is a child of $\rt$. This implies that 
\begin{equation}\label{eq:maximize_rank}
\max_{\substack{u,v\in\TT\\ u\neq v}}\rk_v(u)=\max_{u\in\ch_\TT(\rt)}\rk_{\rt}(u). 
\end{equation}

We next describe how ornamentations behave under $\Pop$. Consider an ornamentation $\de \in \OO(\TT)$. We say a node $v$ \dfn{wraps} a node $u$ in $\de$ if $\de(u)\subseteq\de(v)$ and there does not exist a node $w$ such that $\de(u)\subsetneq\de(w)\subsetneq \de(v)$. 
When $v$ wraps $u$ in $\delta$, we define the \dfn{reduction} of the ornament $\de(v)$ by the node $u$ to be the ornament $\de(v)\setminus\Delta_{\de(v)}(u)$, and we define $\de_u^v\colon\TT\to\Orn(\TT)$ by \[ \de_u^v(w) = \begin{cases} 
          \de(v)\setminus\Delta_{\de(v)}(u) & \text{if } w = v; \\
          
          \de(w) & \text{otherwise}.
       \end{cases}
    \] See \cref{fig:wrapping_reduction}. 

\begin{lemma}\label{lem:wrapping_reduction}
    Let $\de\in\OO(\TT)$, and suppose $v$ wraps $u$ in $\de$. The map $\de_u^v$ is an ornamentation. 
\end{lemma}

\begin{proof}
Suppose instead that $\delta_u^v$ is not an ornamentation. Then there must exist $x\in\TT$ such that $\delta_u^v(v)$ and $\delta_u^v(x)$ are not disjoint or nested. Since $\delta_u^v(v)\subseteq\delta(v)$ and $\delta_u^v(x)=\delta(x)$, it must be the case that $x\in\delta(v)$. Then $\delta(x)$ contains nodes in $\delta_u^v(v)$ and $\Delta_{\delta(v)}(u)$, so it must contain $u$. It follows that $\delta(u)\subseteq\delta(x)\subseteq\delta(v)$. But $x$ cannot equal $u$ or $v$, so this contradicts the assumption that $v$ wraps $u$ in $\delta$. 
\end{proof}

\begin{figure}[ht]
\begin{center}{\includegraphics[height=4.2cm]{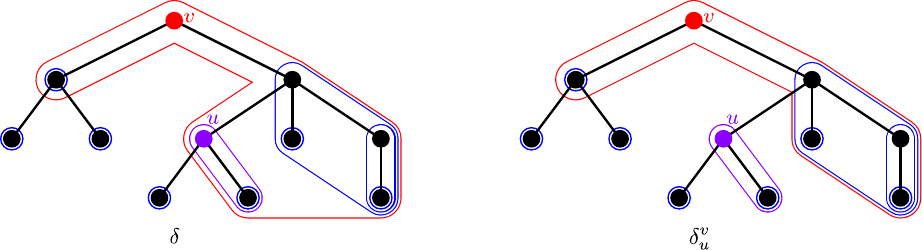}}
\end{center}
\caption{On the left is an ornamentation $\delta$. On the right is the ornamentation $\delta_u^v$ obtained by reducing the ornament hung at $v$ by the node $u$.} 
\label{fig:wrapping_reduction}
\end{figure}

Given an ornamentation $\delta$ and a node $v\in\TT$, define a \dfn{minimal reduction} of the ornament $\delta(v)$ to be a reduction of $\delta(v)$ that is not contained in any other reduction of $\delta(v)$. Intuitively, a (minimal) reduction of $\delta(v)$ is a (minimal) change we can make to an ornamentation without ruining its status as an ornamentation. Let $M_\delta(v)$ be the set of nodes $u\in\delta(v)$ such that $\delta(v)\setminus\Delta_{\delta(v)}(u)$ is a minimal reduction of $\de(v)$. 
We will see that each ornamentation covered by $\de$ in $\OO(\TT)$ differs from $\de$ only by replacing one ornament with a minimal reduction of that ornament. In the following lemma, we describe which reductions of $\de$ are minimal. 

\begin{lemma}\label{rem:ornamentpartition}
    Let $\de \in \OO(\TT)$ and $v \in \TT$, and let $u_1, u_2, \dots, u_k$ be the nodes wrapped by $v$ in $\de$. 
    \begin{enumerate}[(i)]
        \item We have $\de(v) = \{v\} \sqcup \de(u_1) \sqcup \de(u_2) \sqcup \dots \sqcup \de(u_k)$. 
        \item If $o$ is a reduction of $\de(v)$, then $o = \{v\} \cup \bigcup_{u\in U} \de(u)$ for some nonempty $U \subseteq \{u_1, u_2, \dots, u_k\}$.
        \item A reduction $o$ of $\de(v)$ is minimal if and only if $o = \de(v) \setminus \de(u_i)$ for some $i$.
    \end{enumerate}
\end{lemma}
We refer to the ornaments $\de(u_i)$ from \cref{rem:ornamentpartition} as \dfn{maximal subornaments} of $\de(v)$. 

\begin{proof} 
\textit{(i)} This simply asserts that each node in $\de(v)$ is either $v$ or a member of a maximal subornament of $v$. The maximal subornaments are disjoint because if they were nested, the inner ornament would not be wrapped by $v$.

\textit{(ii)} By definition of a reduction, we can write $o = \de(v)\setminus\Delta_{\de(v)}(u_i)$ for some $i \in [k]$. The tree structure and the connectedness of each of the ornaments ensure that $\Delta_{\de(v)}(u_i)$ is the union of the ornament $\de(u_i)$ and any other ornaments hanging below it.

\textit{(iii)} If $o = \de(v) \setminus \de(u_i)$, then $o$ is a minimal reduction of $\de(v)$ by \textit{(i)} and \textit{(ii)}. To prove the converse, suppose $o$ is a minimal reduction of $\delta(v)$. We may assume without loss of generality that $o = \de(v) \setminus (\de(u_1) \sqcup \dots \sqcup \de(u_\ell))$ for some $\ell < k$. There exists $i\in[\ell]$ such that $u_i \not>_\TT u_j$ for all $j\in[k]\setminus\{i\}$. Then $\de(v) \setminus \de(u_i)$ is a reduction of $\de(v)$ that contains $o$. By the minimality of $o$, we have $o=\de(v)\setminus\de(u_i)$. 
\end{proof}

In the following lemma, we describe how $\Pop$ operates on ornamentations. Parts \textit{(i)} and \textit{(ii)} are two sides of the same coin, showing the bijection between minimal reductions and covered elements in the ornamentation lattice. Part \textit{(iii)} explicitly describes how $\Pop$ is computed.

\begin{lemma}\label{lem:popmain}
    Let $\de \in \OO(\TT)$. 
    \begin{enumerate}[(i)]
        \item For every $v\in\TT$ and every $u\in M_\de(v)$, the ornamentation $\de_u^v$ is covered by $\de$ in $\OO(\TT)$.
        \item For every ornamentation $\de'$ covered by $\de$ in $\OO(\TT)$, there exist $v\in \TT$ and $u\in M_\de(v)$ such that $\de'=\de_u^v$.
        \item For all $v \in \TT$, $$\Pop(\de)(v) = \bigcap_{u\in M_\de(v)} \de_u^v(v).$$ 
    \end{enumerate}
\end{lemma}

\begin{proof}
\textit{(i)} We have $\de_u^v < \de$, as $\de_u^v(w)\subseteq\de(w)$ for all $w\in\TT$. Suppose for the sake of contradiction that an ornamentation $\sigma\in \OO(\TT)\setminus\{\de, \de_u^v\}$ satisfies $\de_u^v < \sigma < \de$. Then $\de_u^v(w) = \sigma(w) = \de(w)$ for all $w\in\TT\setminus\{v\}$, so we must have $\de_u^v(v)\subsetneq\sigma(v)\subsetneq\de(v)$. In accordance with \cref{rem:ornamentpartition}, we write $\de(v) = \{v\} \sqcup \de(u_1) \sqcup \de(u_2) \sqcup \dots \sqcup \de(u_k)$, where $u_1, \dots, u_k$ are all the nodes wrapped by $v$ in $\de$. Because $\sigma(v)$ is a proper subset of $\de(v)$ that must either contain or be disjoint from each ornament $\sigma(u_i)=\de(u_i)$, there must be an index $j$ such that $u_j\in\de(v)\setminus\sigma(v)$. Because $\sigma(v)$ is connected and does not include $u_j$, it must be a subset of $\de(v)\setminus\Delta_{\de(v)}(u_j) = \de_{u_j}^v(v)$. Thus, $\de_u^v(v) \subsetneq \sigma(v) \subseteq \de_{u_j}^v(v)$; this is impossible because $\de_{u_j}^v(v)$ is a reduction of $\de(v)$ and $\de_u^v(v)$ is a minimal reduction of $v$. It follows that $\sigma$ cannot exist, so $\de_u^v$ is covered by $\de$.

\textit{(ii)} Let $\de'$ be an ornamentation covered by $\de$. Suppose by way of contradiction that $\de'$ differs from $\de$ in more than one ornament, so $\de'(v_1) \subsetneq \de(v_1)$ and $\de'(v_2) \subsetneq \de(v_2)$ for some distinct $v_1,v_2$. Assume without loss of generality that $\de(v_1)$ is either disjoint from or contained in $\de(v_2)$. Define a map $\sigma\colon\TT\to\Orn(\TT)$ by letting $\sigma(v_1)=\de'(v_1)$ and $\sigma(w)=\de(w)$ for all $w\neq v_1$. Then $\sigma$ is an ornamentation (because $\de'(v_1)$, as a subset of $\de(v_1)$, must either be nested in or disjoint from $\de(v_2)$). But $\de' < \sigma < \de$, contradicting the fact that $\de$ covers $\de'$. Thus, $\de$ and $\de'$ disagree in exactly one ornament. Let $v$ be the unique node of $\TT$ such that $\de'(v)\neq\de(v)$. Using \cref{rem:ornamentpartition} once again, we write $\de(v) = \{v\} \sqcup \de(u_1) \sqcup \de(u_2) \sqcup \dots \sqcup \de(u_k)$, where $u_1, \dots, u_k$ are all the nodes wrapped by $v$ in $\de$. We may assume without loss of generality that $\de'(v) = \{v\} \sqcup \de(u_1) \sqcup \dots \sqcup \de(u_\ell)$ for some $\ell < k$. There must exist an index $i$ such that $u_i \in \de(v)\setminus\de'(v) = \de(u_{l+1})\sqcup\de(u_{l+2})\sqcup\dots\sqcup\de(u_k)$ and $u_i \not>_\TT u_j$ for all $j \in [k]$. (Any minimal element in $\de(v)\setminus\de'(v)$ is also minimal in $\de(v)$ because $\de'(v)$ is a connected subtree with the same maximal element as $\de(v)$.) Then, the minimal reduction $\de^v_{u_i}$ satisfies $\de'(v) \subseteq \de^v_{u_i}(v) \subsetneq \de(v)$, and since $\de'$ is covered by $\de$, we must have $\de'=\de_{u_i}^v$.

\textit{(iii)} This is immediate from \textit{(i)} and \textit{(ii)} and the definition of $\Pop$. 
\end{proof}

\begin{example}
\cref{fig:enter-label} illustrates how to apply $\Pop$ to an ornamentation. Note that although the ornament $o=\de(v)\setminus\Delta_{\de(v)}(w)$ is a reduction of $v$, it is not minimal because $\de(v)\setminus\Delta_{\de(v)}(w')$ is a reduction of $v$ that contains $o$. This is why $\Pop(\delta)(v)$ contains $\delta(w)$. 
\end{example}

\begin{figure}[ht]
    \centering
    \includegraphics[height=4.314cm]{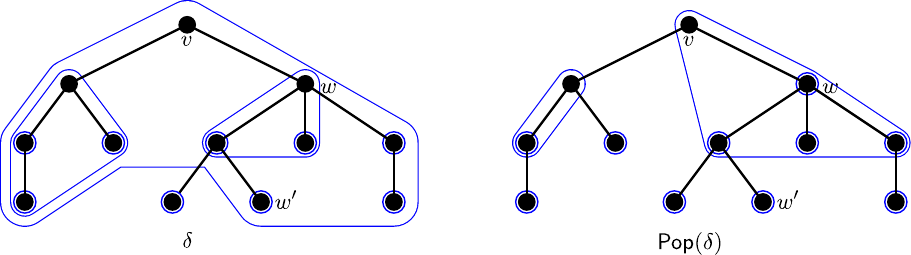}
    \caption{Applying $\Pop$ to the ornamentation on the left yields the ornamentation on the right.
    }
    \label{fig:enter-label}
\end{figure}


Define a \dfn{section} of an ornament $\de(v)$ to be a set of nodes of the form $\Delta_{\de(v)}(u)$ for some ${u\in\ch_{\de(v)}(v)}$. We call $\Delta_{\de(v)}(u)$ the \dfn{$u$-section} of $\de(v)$. Each ornament is the disjoint union of its root and its sections. Also note that a singleton ornament has no sections. The following lemma states that $\Pop$ strictly decreases the size of each section. 
\begin{lemma}\label{rem:monotonic}
    Let $\de \in \OO(\TT)$. For all $v \in \TT$,
    we have $\Pop(\de)(v) \subseteq \de(v)$. Additionally, for each child $u \in \ch_{\de(v)}(v)$, we have $$|\Delta_{\Pop(\de)(v)}(u)| \leq |\Delta_{\de(v)}(u)| - 1.$$ 
\end{lemma}
\begin{proof}
The first claim is immediate since $\Pop(\de)\leq\de$. Let $\de(v) = \{v\} \sqcup \de(u_1) \sqcup \de(u_2) \sqcup \dots \sqcup \de(u_k)$ be the decomposition of $\de(v)$ into its maximal subornaments. Each section $\lambda$ of $\de(v)$ must contain at least one node $w_\lambda\in\{u_1,u_2,\ldots,u_k\}$ that is minimal among $u_1, u_2, \dots, u_k$ in the partial order $\leq_\TT$. Note that $w_\lambda$ must be such that $\de(v)\setminus\de(w_\lambda)$ is a minimal reduction of $\de(v)$. By part \textit{(iii)} of \cref{lem:popmain}, we know that $\Pop(\de)(v)$ does not contain $w_\lambda$. 
\end{proof}

Define the $\dfn{weight}$ of an ornamentation $\de\in\OO(\TT)$ to be $\WT(\de) = \sum_{v\in \TT} |\de(v)|$. If $\de' \leq \de$, then ${\WT(\de') \leq \WT(\de)}$, where equality holds if and only if $\de = \de'$. Thus, if $\de' \leq \de$ and $\WT(\de') = \WT(\de) - 1$, then $\de$ must cover $\de'$.

\section{Maximum Forward Orbit Sizes}\label{sec:maxorbit}

In this section, we prove \cref{thm:forward_orbit}, which determines the maximum possible size of a forward orbit of the pop-stack operator on an ornamentation lattice. Recall that $\rt$ denotes the root of the tree $\TT$. The formula for $\max_{\de\in \OO(\TT)}|\Orb_{\Pop}(\de)|$ stated in \cref{thm:forward_orbit} has the following equivalent expressions:
\begin{equation}\label{eq:equivalent} 
\max_{u\in \ch_\TT(\rt)}\rk_{\rt}(u) + 2 = \max_{C\in\MM_\TT}(|C| + f_C(\rt))
=  \max_{C\in\MM_\TT}\min_{u\in C\setminus\{\rt\}}(|\Delta_\TT(u)| + 2\depth(u) - 1).
\end{equation} 
The first equality follows from the fact that \[\max_{u\in \ch_\TT(\rt)}\rk_\rt(u) + 2 = \max_{u\in \ch_\TT(\rt)}\height_{\TT^*_\rt}(u) + 2 = \height_{\TT^*_\rt}(\rt) + 1\] and the observation that $\height_{\TT^*_\rt}(\rt) + 1$ is equal to the maximum size of a chain in $\TT^*_\rt$. The second equality in \eqref{eq:equivalent} is a straightforward consequence of \eqref{eq:fC2}.

\subsection{The Lower Bound}

Let $\TT$ be a rooted tree with root $\rt$. We will first construct an ornamentation $\de^\dagger\in\OO(\TT)$ whose forward orbit under $\Pop$ has size ${\max_{C\in\MM_\TT}(|C|+f_C(\rt))}$ (which is equal to the claimed maximum forward orbit size by \eqref{eq:equivalent}). Let $C^*$ be a maximal chain of $\TT$ such that $|C^*|+f_{C^*}(\rt)=\max_{C\in\MM_\TT}(|C|+f_C(\rt))$, and let $f_{C^*}(\rt) = k$.

Let $v_0, v_1, \dots, v_{|C^*| - 1}$ be the nodes of $C^*$ ordered from top to bottom; note that $v_0=\rt$. For $0\leq i\leq |C^*|-2$, let $B_i=\{u\in\TT\setminus C^*:u\leq_\TT v_i\text{ and }u\not\leq_\TT v_{i+1}\}$. We call $B_i$ the $i$-th \dfn{bough} of $C^*$. Note that boughs may be empty. Moreover, $B_0,B_1,\ldots,B_{|C^*|-2}$ form a partition of $\TT\setminus C^*$. 

Define a partial order $\sqsubseteq$ on $\TT\setminus C^*$ as follows. Suppose $x\in B_i$ and $y\in B_j$. If $i>j$, then $x\sqsubseteq y$. If $i=j$, then $x\sqsubseteq y$ if and only if $x\geq_\TT y$. Let $v_{|C^*|},v_{|C^*|+1},\ldots,v_{|\TT|-1}$ be a linear extension of the poset $(\TT\setminus C^*,\sqsubseteq)$; that is, $v_{|C^*|},v_{|C^*|+1},\ldots,v_{|\TT|-1}$ is an ordering of $\TT\setminus C^*$ such that $a\leq b$ whenever $v_a\sqsubseteq v_b$. 




\begin{figure}[ht]
\begin{center}{\includegraphics[height=8.986cm]{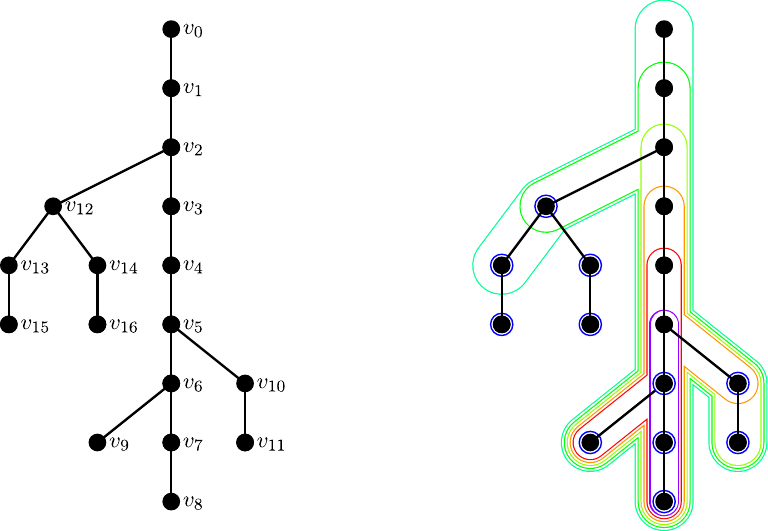}}
\end{center}
\caption{On the left is a tree $\TT$ with a maximal chain $C^*$ consisting of nodes $v_0,v_1,\ldots,v_8$. We have $k=f_{C^*}(\rt)=5$. On the right is the ornamentation $\de^\dagger$. } 
\label{fig:maximum_construction}
\end{figure}

\begin{lemma}\label{lem:descendants} 
    For each nonnegative integer $i<k$, the nodes $v_{|C^*|}, v_{|C^*| + 1}, \dots, v_{|C^*|+k-i-1}$ are descendants of $v_{i+1}$.
\end{lemma}

\begin{proof} 
Recall from \eqref{eq:fC1} that $f_{C^*}$ takes the value $0$ at the bottom of the chain $C^*$ and increases by $0$ or $1$ whenever we move one step up along the chain. Furthermore, since $v_{|C^*|-1}$ is a leaf, we have $b_{C^*}(v_{|C^*|-1})=0$, so $f_{C^*}(v_{|C^*|-2})=0$. It follows that $f_{C^*}(v_i) \leq \height_{C^*}(v_i) - 1$ for all $0\leq i<k$. In particular, $k = f_{C^*}(\rt) \leq |C^*| - 2$. This shows that $v_{i+1}$ actually is an element of $C^*$. By a similar argument, we have $f_{C^*}(v_i) \geq f_{C^*}(v_0) - i = k - i$. We also know from \eqref{eq:fC1} that $f_{C^*}(v) \leq b_{C^*}(\ch_{C^*}(v))$ for all $v \in C^*$. Because $\ch_{C^*}(v_i) = v_{i+1}$, this implies that $b_{C^*}(v_{i+1}) \geq f_{C^*}(v_i) \geq k - i$. Therefore, it follows from the definition of $b_{C^*}$ that there are at least $k - i$ descendants of $v_{i+1}$ that do not lie on $C^*$. In other words, $|B_{i+1}\sqcup B_{i+2}\sqcup\cdots\sqcup B_{|C^*|-2}|\geq k-i$. Since $v_{|C^*|}, v_{|C^*| + 1}, \dots, v_{|\TT|-1}$ is a linear extension of $(\TT\setminus C^*,\sqsubseteq)$, the desired result follows from the definition of $\sqsubseteq$.  
\end{proof}

Now, define $\de^\dagger\colon\TT\to\Orn(\TT)$ as follows. For $0\leq i\leq k$, let $\de^\dagger(v_i)=\{v_i,v_{i+1},\ldots,v_{|C^*|+k-1-i}\}$; for $v\not\in\{v_0,v_1,\ldots,v_k\}$, let $\de^\dagger(v)=\{v\}$. See \cref{fig:maximum_construction} for an example. Note that the definition of the partial order $\sqsubseteq$ on nodes within the same bough 
ensures that each set $\delta^\dagger(v_i)$ induces a connected subgraph of $\TT$. Using \cref{lem:descendants}, one can readily check that $\de^\dagger$ is an ornamentation. 

\begin{lemma}\label{lem:Pop_delta_dagger}
For all integers $0\leq i\leq k$ and $p\geq 0$, we have \[\Pop^p(\delta^\dagger)(v_i)=\{v_i\}\cup\{v_j:i+1\leq j\leq |C^*|+k-1-i-p\}.\] 
\end{lemma}

\begin{proof} 
If $p=0$, then the desired result is immediate from the definition of $\delta^\dagger$. Thus, we may assume $p\geq 1$ and proceed by induction on $p$. If $i+1>|C^*|+k-1-i-p$, then we know by induction that $\Pop^{p-1}(\de^\dagger)(v_i)\subseteq\{v_i,v_{i+1}\}$, so it follows from \cref{rem:monotonic} that $\Pop^p(\de^\dagger)(v_i)=\{v_i\}$, as desired. Now assume $i+1\leq |C^*|+k-1-i-p$. For convenience, let $m=|C^*|+k-i-p$. By induction, we have $\Pop^{p-1}(\de^\dagger)(v_i)=\{v_i,v_{i+1},\ldots,v_m\}$, $\Pop^{p-1}(\de^\dagger)(v_{i+1})=\{v_{i+1},\ldots,v_{m-1}\}$, and $\Pop^{p-1}(\de^\dagger)(v_m)=\{v_{m}\}$. It follows that the only ornaments wrapped by $\Pop^{p-1}(\de^\dagger)(v_i)$ in $\Pop^{p-1}(\de^\dagger)$ are $\{v_{i+1},\ldots,v_{m-1}\}$ and $\{v_m\}$. If $i+1\leq m\leq |C^*|-1$, then $v_m$ is a descendant of $v_{i+1}$ because both $v_{i+1}$ and $v_{m}$ are in the chain $C^*$. On the other hand, if $m\geq|C^*|$, then it follows from \cref{lem:descendants} that $v_m$ is a descendant of $v_{i+1}$. In either case, $v_m$ is a descendant of $v_{i+1}$, so we deduce that the reduction of $\Pop^{p-1}(\de^\dagger)(v_i)$ by the node $v_{i+1}$ is not a minimal reduction. Hence, $M_{\Pop^{p-1}(\de^\dagger)}(v_i)=\{v_m\}$. It follows from \cref{lem:popmain} that \[\Pop^p(\de^{\dagger})(v_i)=\Pop^{p-1}(\de^\dagger)(v_i)\setminus\{v_m\}=\{v_i,\ldots,v_{m-1}\},\] as desired. 
\end{proof} 

It follows from \cref{lem:Pop_delta_dagger} that $\Pop^{|C^*|+k-1}(\delta^\dagger)(v)=\{v\}$ for all $v\in\TT$, so $\Pop^{|C^*|+k-1}(\de^\dagger)$ is the minimum element of $\OO(\TT)$. This lemma also tells us that $\Pop^{|C^*|+k-2}(\de)(v_0) = \{v_0, v_1\}$. Hence, 
\begin{equation}\label{eq:delta_dagger_maximizes}
|\Orb_{\Pop}(\de^\dagger)|=|C^*|+k=\max_{C\in\mathcal M_{\TT}}(|C|+f_C(\rt)). 
\end{equation}

\subsection{Upper bound}

We now seek to derive a tight upper bound for $\max_{\de\in \OO(\TT)}|\Orb_{\Pop}(\de)|$. To do this, we will first prove that for $\de\in\Pop^k(\OO(\TT))$ and $v\in\TT$, every node in $\de(v)$ has $v$-rank at least $k$. When $k$ is strictly greater than the maximum of $\rk_u(v)$ over all distinct $u,v\in\TT$, this will imply that the only ornamentation in $\Pop^k(\OO(\TT))$ is the minimum ornamentation $\de_{\min}$. Therefore, in light of \eqref{eq:maximize_rank}, it will follow that 
\[\max_{\de\in \OO(\TT)}|\Orb_{\Pop}(\de)|\leq\max_{u\in\ch_\TT(\rt)}\rk_{\rt}(u)+2.\] Combining this with \eqref{eq:equivalent} and \eqref{eq:delta_dagger_maximizes} will then complete the proof of \cref{thm:forward_orbit}. 

\begin{proposition}
\label{lem:vrank}
 Let $\de \in \Pop^k(\OO(\TT))$. For all $v\in\TT$ and $u\in\de(v)$, we have $\rk_v(u)\geq k$. 
\end{proposition}

\begin{proof}
The proof is trivial when $k=0$, so we may assume $k\geq 1$ and proceed by induction on $k$. Since $\de\in\Pop^{k}(\OO(\TT))$, we can write $\de=\Pop(\de')$ for some $\de'\in\Pop^{k-1}(\OO(\TT))$. Suppose $u,v\in\TT$ are such that $\rk_v(u)\leq k-1$; our aim is to show that $u\not\in\de(v)$. This is certainly true if $u\not\in\de'(v)$ since $\de(v)\subseteq\de'(v)$, so we may assume $u\in\de'(v)$. By induction, this implies that $\rk_v(u)=k-1$. 

It is immediate from the definition of $v$-rank that $\rk_v(u')<\rk_v(u)$ for all descendants $u'$ of $u$. By induction, none of the descendants of $u$ are in $\de'(v)$. Hence, $u$ is a minimal element of $\de'(v)$. Furthermore, $\de'(u)$ must be nested in $\de'(v)$, so $\de'(u)=\{u\}$. If $v$ wraps $u$ in $\de'$, then $u\in M_{\de'}(v)$, so it follows from \cref{lem:popmain} that $u\not\in\Pop(\de')(v)=\de(v)$, as desired. Hence, we may assume that $v$ does not wrap $u$ in $\de'$. This means that there exists a node $v'$ such that $u<_\TT v'<_\TT v$ and $u\in\de'(v')$. 



By the induction hypothesis, we have $\rk_{v'}(u) \geq k-1$. Referring to \cref{vrankdef}, we find that ${\rk_v(u) = \height_{\TT^*_v}(u) = k-1}$ and $\rk_{v'}(u) = \height_{\TT^*_{v'}}(u) \geq k-1$. Using the definitions of height and of the trees $\TT^*_v$ and $\TT^*_{v'}$, we find that
\begin{equation}\label{eq:height1} \height_{\TT^*_v}(u) = \max_{\substack{C\in\MM_{\TT^*_v} \\ C\ni u}}\height_C(u)=\max_{\substack{C\in\MM_{\TT} \\ C\ni u}}(\height_C(u) + f_C(v)) = k-1  
\end{equation} 
and 
\begin{equation}\label{eq:height2}
\height_{\TT^*_{v'}}(u) = \max_{\substack{C\in\MM_{\TT^*_{v'}} \\ C\ni u}}\height_C(u)=\max_{\substack{C\in\MM_{\TT} \\ C\ni u}}(\height_C(u) + f_C(v')) \geq k-1. 
\end{equation}

Let $C^*\in\MM_\TT$ be a maximal chain containing $u$ such that \[\height_{C^*}(u)+f_{C^*}(v')=\max_{\substack{C\in\MM_{\TT} \\ C\ni u}}(\height_C(u) + f_C(v')).\] Since the values of $f_{C^*}$ weakly increase as we move up the chain $C^*$, we have ${f_{C^*}(v')\leq f_{C^*}(v)}$. It follows from \eqref{eq:height1} and \eqref{eq:height2} that 
\[k-1 \leq \height_{C^*}(u) + f_{C^*}(v') \leq \height_{C^*}(u) + f_{C^*}(v) \leq \max_{\substack{C\in\MM_\TT \\ C\ni u}}\height_C(u) + f_C(v) = k-1,\] so we have $k-1 = \height_{C^*}(u) + f_{C^*}(v)$ and $f_{C^*}(v)=f_{C^*}(v')$. Since, once again, the values of $f_{C^*}$ weakly increase as we move up $C^*$, it follows that $f_{C^*}(v)=f_{C^*}(w)$, where $w=\ch_{C^*}(v)$. Recalling the definition of $f_{C^*}$ from \eqref{eq:fC1}, we deduce that $f_{C^*}(v)\geq b_{C^*}(w)$. Consequently, 
\begin{align}
\nonumber k-1&=\height_{C^*}(u)+f_{C^*}(v) \\ 
\nonumber &\geq\height_{C^*}(u)+b_{C^*}(w) \\ 
\label{eq:k-1_height}&=\height_{C^*}(u) + (|\Delta_\TT(w)| - \height_{C^*}(w) - 1). 
\end{align}

Now, suppose by way of contradiction that $u\in\de(v)$. The set $\Delta_{\de(v)}(w)$ induces a connected subgraph of $\TT$ that contains both $u$ and $w$, so 
\begin{equation}\label{eq:done}
|\Delta_{\delta(v)}(w)| \geq \height_{C^*}(w) - \height_{C^*}(u) + 1.
\end{equation} 
Since $\de\in\Pop^k(\OO(\TT))$, we can write $\de=\Pop^k(\sigma)$ for some ornamentation $\sigma$. We know by \cref{rem:monotonic} that 
\[|\Delta_{\Pop^{i+1}(\sigma)(v)}(w)|\leq|\Delta_{\Pop^i(\sigma)(v)}(w)|-1\] for all $0\leq i\leq k-1$. Therefore, 
\[|\Delta_{\de(v)}(w)|=|\Delta_{\Pop^{k}(\sigma)(v)}(w)|\leq |\Delta_{\sigma(v)}(w)|-k\leq|\Delta_\TT(w)|-k.\] But \eqref{eq:k-1_height} tells us that 
\[|\Delta_\TT(w)|-k\leq\height_{C^*}(w)-\height_{C^*}(u),\] and this contradicts \eqref{eq:done}. 
%
\end{proof} 



As mentioned above, \eqref{eq:equivalent}, \eqref{eq:delta_dagger_maximizes}, and \cref{lem:vrank} together imply \cref{thm:forward_orbit}. 

\section{The Image of $\Pop$}\label{sec:image}  
Our goal in this section is to characterize the image of the pop-stack operator on the ornamentation lattice $\OO(\TT)$. This characterization, which we describe in \cref{cla:imageOfPop}, is given by imposing a simple local criterion on each ornament in an ornamentation.

Let $\TT \in \PT_n$ and $\de \in \OO(\TT)$. For convenience, we add to $\TT$ an extra \emph{imaginary node} $\omega$, and we let $\de(\omega) = \{\omega\}\cup\TT$. We make the convention that $\omega$ is the parent of the root $\rt$ and that $\rt$ is the unique child of $\omega$. We impose that the ornament hung at $\omega$ is left unaffected by $\Pop$. As before, we say a node $v$ \dfn{wraps} a node $u$ in $\de$ if $\de(u)\subseteq\de(v)$ and there does not exist a node $w$ such that $\de(u)\subsetneq\de(w)\subsetneq \de(v)$. We then say that $v$ \dfn{hugs} $u$ in $\de$ if $v$ wraps $u$ and there exists a child $u' \in \ch_{\de(u)}(u)$ such that $\Delta_{\de(u)}(u') = \Delta_{\de(v)}(u')$. 
Note that nodes whose ornaments are singletons can not be hugged, and that a node may be hugged by the imaginary node $\omega$ under our definition. We do not draw $\omega$ in diagrams. 

\begin{lemma}\label{lem:disjointornested}
    Let $\de \in \OO(\TT)$. For any two nodes $u, v \in \TT\cup\{\omega\}$, the ornaments $\de(u)$ and $\Pop(\de)(v)$ are either nested or disjoint. 
\end{lemma}
\begin{proof} 
If either $u$ or $v$ is the imaginary node $\omega$, then either $\de(u)$ or $\Pop(\de)(v)$ is $\TT\cup\{\omega\}$, so the desired result is immediate. Thus, we may assume $u\neq\omega$ and $v\neq\omega$. By the definition of an ornamentation, $\de(u)$ and $\de(v)$ are either nested or disjoint, and by \cref{rem:monotonic}, $\Pop(\de)(v) \subseteq \de(v)$. Thus, if $\de(v)$ is disjoint from or contained in $\de(u)$, then its subset $\Pop(\de)(v)$ is also disjoint from or contained in $\de(u)$. Now suppose $\de(u)\subseteq\de(v)$. There exists a node $v'$ such that $\de(u)\subseteq\de(v')\subseteq\de(v)$ and such that $v$ wraps $v'$ in $\de$. Now, \cref{rem:ornamentpartition} implies that any minimal reduction $o$ of $\de(v)$ will either contain or be disjoint from $\de(v')$, and hence will either contain or be disjoint from $\de(u)$. According to \cref{lem:popmain}, $\Pop(\de)(v)$ either contains or is disjoint from $\de(u)$, as desired. 
\end{proof}

The following theorem is equivalent to \cref{thm:pop_image_intro}, but it is phrased using the language of hugs. 

\begin{theorem}\label{cla:imageOfPop}
    Let $\de \in \OO(\TT)$. Then $\de \in \Pop(\OO(\TT))$ if and only if each $u \in \TT$ is not hugged in $\de$ by any node in $\TT\cup\{\omega\}$. 
\end{theorem}
\begin{proof} 
Suppose $\de = \Pop(\de') \in \Pop(\OO(\TT))$, and assume for the sake of contradiction that there exist $u \in \TT$ and $v \in \TT\cup\{\omega\}$ such that $v$ hugs $u$ in $\de$. Since $v$ hugs $u$, there exists $u' \in \ch_{\de(u)}(u)$ such that $\Delta_{\de(u)}(u') = \Delta_{\de(v)}(u')$.
\cref{rem:monotonic} implies that $|\Delta_{\de'(u)}(u')| > |\Delta_{\de(u)}(u')| = |\Delta_{\de(v)}(u')|$, so $\de'(u) \not\subseteq \de(v)$. Additionally, $u' \in \de'(u) \cap \de(v)$ and $v \in \de(v) \setminus \de'(u)$. Thus, $\de'(u)$ and $\de(v)$ are neither disjoint nor nested. This contradicts \cref{lem:disjointornested} because $\de(v) = \Pop(\de')(v)$. 

To prove the converse, suppose now that each node $u \in \TT$ is not hugged by any node in $\TT\cup\{\omega\}$. For each $u\in\TT$, let $\de^*(u)=\Delta_{\de(v)}(u)$, where $v\in\TT\cup\{\omega\}$ is the unique node that wraps $u$ in $\de$. We will show that $\Pop(\de^*) = \de$. Fix $u\in\TT$; we will show that $\Pop(\de^*)(u)=\de(u)$. 


\begin{figure}[ht]\label{fig:inverse}
    \centering
    \includegraphics[height=4.223cm]{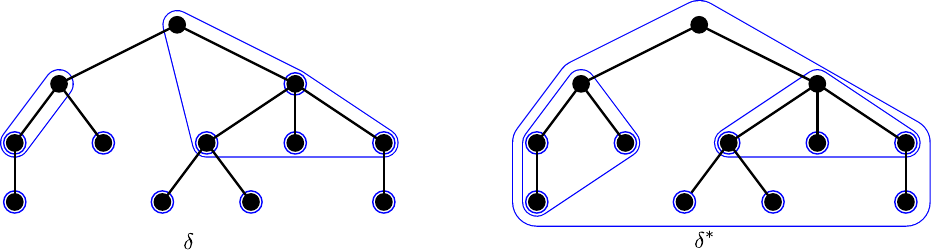}
    \caption{An ornamentation $\de$ and the ornamentation $\de^*$ constructed in the proof of \cref{cla:imageOfPop}.} 
\end{figure}

First, we show that $\Pop(\de^*)(u) \subseteq \de(u)$. Let $v$ be the node that wraps $u$ in $\de$. Consider a node $w \not\in \de(u)$; we wish to show that $w\not\in\Pop(\de^*)(u)$. If $w \not\in \de^*(u)$, then this is immediate from \cref{rem:monotonic}, so we may assume that $w \in \de^*(u)\setminus\de(u)$. Let $t$ be first node on the path from $u$ to $w$ that is not in $\de(u)$. Then $w$ is a descendant of $t$, and $t$ must be a node wrapped by $v$ in $\de$. It follows that $\de^*(t) = \Delta_{\de(v)}(t)$ and $\de^*(u) = \Delta_{\de(v)}(u)$, so $\de^*(u)\setminus\de^*(t)$ is a minimal reduction of $\de^*(u)$. By part \textit{(iii)} of \cref{lem:popmain}, this implies that $\Pop(\de^*)(u)$ is disjoint from $\de^*(t)$. Since $w\leq_\TT t<_\TT u$ and $t\not\in\Pop(\de^*)(u)$, it follows that $w \not\in \Pop(\de^*)(u)$.


Next, we show that $\de(u) \subseteq \Pop(\de^*)(u)$. This is trivial if $\de(u) = \{u\}$, so suppose $|\de(u)| > 1$. Let $w \in \de(u) \setminus \{u\}$; we will show that $w \in \Pop(\de^*)(u)$. Let $u'$ be the child of $u$ such that $w\leq_\TT u'$. Let $v$ be the node that wraps $u$ in $\de$. 
Because $v$ does not hug $u$ by assumption, $\Delta_{\de(u)}(u') \subsetneq \Delta_{\de(v)}(u')$. Therefore, there is a node $u'' \in \Delta_{\de(v)}(u') \setminus \Delta_{\de(u)}(u')$ that is wrapped by $v$ in $\de$. By definition of $\de^*$, we have $\de^*(u) = \Delta_{\de(v)}(u)$ and $\de^*(u') = \Delta_{\de(u)}(u')$, so $\de^*(u')$ contains $w$. It is not difficult to see that $u$ wraps $u'$ and $u''$ in $\de^*$, so $\de^*(u')$ and $\de^*(u'')$ are maximal subornaments of $\de^*(u)$. 
Assume for the sake of contradiction that $w \not\in \Pop(\de^*)(u)$. By part \textit{(iii)} of \cref{lem:popmain}, there exists a minimal reduction $o$ of $\de^*(u)$ that does not contain $w$. By part \textit{(ii)} of \cref{rem:ornamentpartition}, $o$ is a disjoint union of maximal subornaments of $\de^*(u)$ not including $\de^*(u')$ (because it contains $w$) and not including $\de^*(u'')$ (as $u''$ is a descendant of $u'$ and $o$ must be connected). But by part \textit{(iii)} of \cref{rem:ornamentpartition}, this contradicts the minimality of $o$. 
\end{proof}

\section{The Image of $\Pop^k$}\label{sec:image_iterate} 



Now that we have characterized the image of $\Pop$ on an arbitrary ornamentation lattice, we turn to studying the images of higher iterates $\Pop$. This turns out to be much more complicated. \cref{lem:vrank} gives a necessary condition in terms of $v$-ranks for an ornamentation to lie in the image of $\Pop^k$. The $v$-ranks are fairly global, as determining the $v$-rank of a node could require checking every node in its subtree. We will provide another necessary (but not sufficient) condition for an ornamentation to lie in the image of $\Pop^k$, this time depending on features that are local to each ornament. We will then use this condition to completely characterize $\Pop^k(\OO(\TT))$ when $\TT$ is a chain (so $\OO(\TT)$ is a Tamari lattice).

Let $\TT \in \PT_n$ and $\de \in \OO(\TT)$, and fix a node $v \in \TT$. 
Letting $\lambda_{v'}$ be the $v'$-section of $\de(v)$ (for some $v' \in \ch_{\de(v)}(v)$), an ornament $\de(u)\in\TT$ is a \dfn{bead} of $\lambda_{v'}$ in $\de$ if $u \leq_\TT v'$, $u\not\in\de(v)$, and for each node $w \in [v', u]$, either $w \in \de(v)$ or $\de(w)$ is a singleton. (In particular, $\de(u)$ must be a singleton to be a bead.)  

\begin{example}
Let $\de$ be the ornamentation of the tree $\TT$ shown in \cref{fig:enter-label2}. The ornament $\de(t)$ has a section with $1$ bead and a section with $3$ beads. The ornaments $\de(u)$ and $\de(v)$ each have exactly one section, which has $3$ beads. The ornament $\de(w)$ has one section with $4$ beads. 
\end{example}

\begin{figure}[ht]
    \centering
    \includegraphics[height=8.45cm]{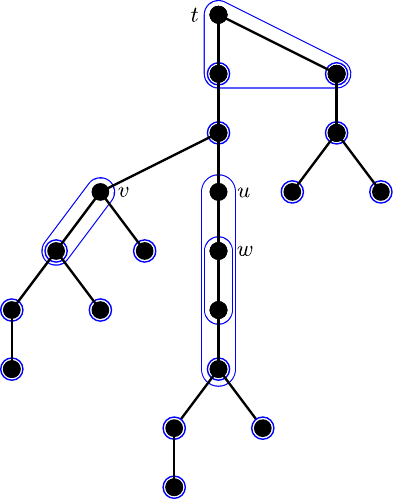}
    \caption{An ornamentation.}
    \label{fig:enter-label2}
\end{figure}

Recall from \cref{rem:ornamentpartition} that a minimal reduction of an ornament $\de(v)$ must be of the form $\de(v)\setminus\de(u)$ for some maximal subornament $\de(u)$. The following lemma describes a special property of the applications of $\Pop$ to an ornamentation that is already in $\Pop(\OO(\TT))$:

\begin{lemma}\label{lem:beads}
    Let $\de \in \Pop(\OO(\TT))$, and let $u\in M_\de(v)$. Then $\de(u)=\{u\}$. Moreover, if $w \in \ch_{\de(v)}(v)$ and $u$ belongs to the $w$-section of $\de(v)$, then for every $k\geq 1$ such that $w\in\Pop^k(\de)(v)$, the set $\{u\}$ is a bead of the $w$-section of $\Pop^k(\de)(v)$ in $\Pop^k(\de)$.
\end{lemma}

\begin{proof}
If $|\de(u)| > 1$, then since $u\in M_\de(v)$, it must be the case that $v$ hugs $u$ in $\de$. This is impossible by \cref{cla:imageOfPop}, so we must have $\de(u)=\{u\}$. 
By \cref{lem:popmain} and induction on $k$, the second part of the lemma follows from the first. 
\end{proof}

\begin{proposition}\label{cla:popk_beads}
    Suppose $\de' =\Pop^k(\de)\in \Pop^k(\OO(\TT))$ for some $\de \in \OO(\TT)$. Then for each $v \in \TT$, each section of $\de'(v)$ has at least $k-1$ beads.
\end{proposition}

\begin{proof}
Suppose $w \in \ch_{\de'(v)}(v)$. For $0\leq i \leq k$, let $\lambda_w^{(i)}$ be the $w$-section of $\Pop^i(\de)(v)$ in $\Pop^i(\de)$. For each $1\leq i\leq k-1$, let $u_i\in M_{\Pop^i(\de)}(v)$ be such that $u_i \in \lambda_w^{(i)}$. It follows from \cref{lem:popmain} that $\Pop^i(\de)$ contains $u_i,\ldots,u_{k-1}$ but does not contain $u_{i-1}$; hence, $u_1,\ldots,u_{k-1}$ are all distinct. According to \cref{lem:beads}, $u_1,\ldots,u_{k-1}$ are all beads of the $w$-section of $\de'(v)$ in $\de'$. 
\end{proof}

While \cref{lem:vrank,cla:imageOfPop,cla:popk_beads} each impose necessary, independent conditions for an ornamentation to belong to $\Pop^k$, their conjunction is still not a sufficient condition. Consider the ornamentation $\de$ in \cref{fig:notsufficient}. None of these three results eliminate the possibility that $\de$ is in the image of $\Pop^2$. Indeed, each node in $\de(v_0)$ has a $v_0$-rank of at least $2$, the unique non-singleton ornament $\de(v_0)$ is not hugged by any ornament (including $\de(\omega)$), and the unique section of $\de(v_0)$ in $\de$ has $2$ beads as necessary. However, it is straightforward to check that $\de$ is not in the image of $\Pop^2$.


    \begin{figure}[ht]
    \begin{center}{\includegraphics[height=3.466cm]{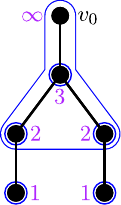}}
    \end{center}
    \caption{An ornamentation demonstrating that \cref{lem:vrank}, \cref{cla:imageOfPop}, and \cref{cla:popk_beads} are insufficient for guaranteeing that an ornamentation belongs to the image of $\Pop^2$. Each node is labeled by its $v_0$-rank.} 
    \label{fig:notsufficient}
    \end{figure} 

The next theorem says that \cref{lem:vrank,cla:imageOfPop,cla:popk_beads} do provide a complete characterization of $\Pop^k(\OO(\TT))$ when $\TT=\CC_n$ is a chain with $n$ elements. We will need a bit more notation. Let $v_1,\ldots,v_n$ be the elements of $\CC_n$, with $v_n<_{\CC_n} v_{n-1}<_{\CC_n}\cdots<_{\CC_n} v_1$. For $\de\in\OO(\CC_n)$ and $i\in[n]$, let $g_\de(i)=\max\{r\in[n]:v_r\in\de(v_i)\}$ so that $\de(v_i)=\{v_j:i\leq j\leq g_\de(i)\}$. This yields a function $g_\de\colon [n] \to[n]$. For $i < j$, the ornaments $\de(v_i)$ and $\de(v_j)$ are disjoint if $g_\de(i) < j$, and they are nested if $g_\de(i) \geq g_\de(j)$. 

\begin{theorem}\label{thm:popk_tamari}
Let $\CC_n$ be a chain with $n$ elements, and let $\de \in \OO(\CC_n)$. Then $\de \in \Pop^k(\OO(\CC_n))$ if and only if for each $v \in \CC_n$, either $\de(v) = \{v\}$ or $v$ satisfies the following properties:
    \begin{enumerate}[(i)]
        \item The node $v$ is not hugged in $\de$ by any node in $\CC_n$.
        \item There are at least $k$ ornaments below $\de(v)$, the top $k-1$ of which are singletons. 
    \end{enumerate}
\end{theorem}
\begin{proof}
We know by  \cref{cla:imageOfPop,cla:popk_beads,rem:monotonic} that the conditions in the statement of the theorem are necessary, so it remains to show they are sufficient. Thus, suppose that for each $v\in\CC_n$, either $\de(v)=\{v\}$ or $v$ satisfies \textit{(i)} and \textit{(ii)}. 
For $0\leq t \leq k-1$, let $\de^{(t)}$ be the ornamentation resulting from extending each non-singleton ornament in $\de$ downwards by $t$ nodes. Formally, $\de^{(t)}$ is the ornamentation whose corresponding function $g_{\de^{(t)}}$ is given by
    \begin{equation*} g_{\de^{(t)}}(i) = \begin{cases} 
          i & \text{if }g_{\de}(i)=i; \\
          g_{\de}(i)+t& \text{otherwise.}
       \end{cases} 
    \end{equation*}
    Note that the ornamentations $\de^{(0)},\ldots,\de^{(k-1)}$ all have the same set of non-singleton ornaments. Also, if $\de^{(k-1)}(u)$ is not a singleton, then $u$ is not hugged in $\de^{(k-1)}$ by any other node in $\CC_n\cup\{\omega\}$. Therefore, \cref{cla:imageOfPop} implies that $\de^{(k-1)}$ is in the image of $\Pop$. Furthermore, it is not difficult to check using \cref{lem:popmain} that for each $t < k-1$, we have $\Pop(\de^{(t+1)})=\de^{(t)}$. It follows that $\de=\Pop^{k-1}(\de^{(k-1)})$, so $\de$ is in the image of $\Pop^k$. 
    \end{proof}

The characterization above allows us to prove \cref{thm:gf}, which states that \[\sum_{n\geq 0}|\Pop^k(\OO(\CC_n))|x^n=\frac{1-x^{k+1}-\sqrt{(1-x^{k+1})^2-4x(1-x)(1-x^{k+1})}}{2x(1-x)}.\]  

\begin{proof}[Proof of \cref{thm:gf}]
Let $A_{n}^{(k)} = |\Pop^k(\OO(\CC_n))|$.
Fix integers $n,k\geq 0$. It follows from \cref{thm:forward_orbit} that $A^{(k)}_n = 1$ if $n \leq k+1$. (This also follows from \cite[Theorem~4.13]{DefantMeeting}.) Now suppose $n\geq k+2$; we will show that 
    \[A^{(k)}_n=\sum_{i=k}^{n-1}A^{(k)}_iA^{(k)}_{n-1-i}.\] 
Using this recurrence, it is routine to derive the desired expression for $\sum_{n\geq 0}A_n^{(k)}x^n$. 

Let $A^{(k)}_{n,a}$ be the number of ornamentations $\de$ in $\Pop^k(\OO(\CC_n))$ such that $g_\de(v_1)=a$. If $a = 1$, the ornament $\de(v_1)$ satisfies the condition in \cref{thm:popk_tamari} because it is a singleton, and does not interact with any of the other ornaments. Hence, $A^{(k)}_{n,a} = A^{(k)}_{n-1}$.

Now suppose $a > 1$. For $\de$ to satisfy the conditions in \cref{thm:popk_tamari}, we must have $a \leq n - k$, and the ornaments $\de(v_a),\de(v_{a+1}),\ldots,\de(v_{a+k-1})$ must be singletons. Consequently, the ornaments $\de(v_i)$ with $i\in[a+k-1]$ must be disjoint from the ornaments $\de(v_j)$ with $j\in[a+k,n]$. This means that to construct $\de$, we must choose the ornaments $\de(v_j)$ with $j\in[a+k,n]$ and then independently choose the ornaments nested in $\de(v_1)$. The number of ways to choose the ornaments $\de(v_j)$ with $j\in[a+k,n]$ is $A^{(k)}_{n-(a+k-1)}$. It remains to choose the ornaments that are nested inside $\de(v_1)$. The first constraint in \cref{thm:popk_tamari} imposes that any non-singleton ornament nested in $\de(v_1)$ cannot contain $v_a$. We can define a bijection between configurations of ornaments nested in $\de(v_1)$ (that satisfy the conditions in \cref{thm:popk_tamari}) and ornamentations in $\Pop^k(\OO(\CC_{a+k-2}))$ as follows. For each such configuration, remove the node $v_1$ and its ornament, and remove all nodes and ornaments strictly below $v_{a+k-1}$. 
This shows that there are $A^{(k)}_{a+k-2}$ ways to choose the ornaments nested in $\de(v_1)$, and we deduce that $A^{(k)}_{n,a}=A^{(k)}_{n-(a+k-1)}A^{(k)}_{a+k-2}$. In conclusion,
\[A^{(k)}_n=A^{(k)}_{n-1}+\sum_{a=2}^{n-k}A^{(k)}_{n-(a+k-1)}A^{(k)}_{a+k-2}=\sum_{i=k}^{n-1}A^{(k)}_iA^{(k)}_{n-1-i}. \qedhere\] 
\end{proof}

\section{Semidistributivity}\label{sec:semidistributivity}  
A finite lattice $L$ is called \dfn{semidistributive} if for all $a,b\in L$ with $a\leq b$, the set ${\{z\in L:z\wedge b=a\}}$ has a maximum and the set     $\{z\in L:z\vee a=b\}$ has a minimum. Associated to a finite semidistributive lattice is a certain simplicial complex called its \emph{canonical join complex}, which has one face for each element of the lattice (see, e.g., \cite{Barnard}). Defant and Williams \cite{Semidistrim} found that elements in the image of $\Pop$ correspond naturally to the facets of the canonical join complex. 

In this brief section, we prove that ornamentation lattices are semidistributive, which provides further motivation for our analysis of the image of $\Pop$ in \cref{sec:image}. We will need the following result due to Barnard. Given a finite poset $P$ with partial order $\leq_P$ and an element $x\in P$, we extend our earlier notation by writing 
\[\Delta_{P}(x)=\{x'\in P: x'\leq_P x\}\quad\text{and}\quad\nabla_P(x)=\{x'\in P:x\leq_P x'\}.\] 

\begin{proposition}[{\cite[Proposition~22]{Barnard}}]\label{prop:Barnard}
A finite lattice $L$ is semidistributive if and only if for every cover relation $x'\lessdot x$ in $L$, the set $\Delta_L(x)\setminus\Delta_L(x')$ has a minimum element and the set $\nabla_L(x')\setminus\nabla_L(x)$ has a maximum element. 
\end{proposition}

\begin{theorem}
Let $\TT$ be a rooted plane tree. The ornamentation lattice $\OO(\TT)$ is semidistributive. 
\end{theorem}

\begin{proof}
For convenience, let $L=\OO(\TT)$. Let $\de'\lessdot\de$ be a cover relation in $L$. By \cref{lem:popmain}, there exist $v\in\TT$ and $u\in M_\de(v)$ such that $\de'=\de_u^v$. Appealing to \cref{prop:Barnard}, we find that we just need to show that $\Delta_L(\de)\setminus\Delta_L(\de_u^v)$ has a minimum element and that $\nabla_L(\de_u^v)\setminus\nabla_L(\de)$ has a maximum element. 

Let $o=\{w\in\TT: u\leq_\TT w\leq_\TT v\}$ be the smallest ornament of $\TT$ that contains both $u$ and $v$. Let $\de_\downarrow$ be the ornamentation of $\TT$ defined so that $\de_\downarrow(w)=\{w\}$ for all $w\in\TT\setminus\{v\}$ and $\de_\downarrow(v)=o$. If $\sigma\in\Delta_L(\de)\setminus\Delta_L(\de_u^v)$, then $\sigma(v)$ must contain a node in $\Delta_{\de(v)}(u)$, so it must in fact contain $o$. It follows that $\de_\downarrow$ is the minimum element of $\Delta_L(\de)\setminus\Delta_L(\de_u^v)$. 

Let $\de^\uparrow$ be the ornamentation defined by 
\[\de^\uparrow(w)=\begin{cases} 
          \Delta_\TT(w)\setminus\Delta_\TT(u) & \text{if $u<_\TT w\leq_\TT v$;} \\
          \Delta_\TT(w) & \text{otherwise.}
       \end{cases} \]
It is immediate from the definition of $\de_u^v$ that $\de^\uparrow$ is in the set $\nabla_L(\de_u^v)\setminus\nabla_L(\de)$; we claim that it is the maximum element of this set. To see this, suppose $\sigma\in\nabla_L(\de_u^v)\setminus\nabla_L(\de)$. Because $\sigma\not\geq\de$, we know that $\sigma(v)$ cannot contain $\de(v)$. But $\de(v)=\de_u^v(v)\sqcup\de(u)$. We have $\de(u)=\de_u^v(u)\subseteq\sigma(u)$ and $\de_u^v(v)\subseteq\sigma(v)$. This implies that $\sigma(v)$ cannot contain $\sigma(u)$, 
so it cannot contain $u$. Hence, $\sigma(v)\subseteq\Delta_\TT(v)\setminus\Delta_\TT(u)$. This implies that $\sigma(w)\subseteq\Delta_\TT(w)\setminus\Delta_\TT(u)$ for all $w$ satisfying $u<_\TT w\leq_\TT v$. Hence, $\sigma\leq\de^\uparrow$. 
\end{proof} 

\section*{Acknowledgments}
Colin Defant was supported by the National Science Foundation under Award No.\ 2201907 and by a Benjamin Peirce Fellowship at Harvard University.


\begin{thebibliography}{99}

\bibitem{Asinowski}
A. Asinowski, C. Banderier, and B. Hackl. Flip-sort and combinatorial aspects of pop-stack sorting. \emph{Discrete Math. Theor. Comput. Sci.}, {\bf 22} (2021). 

\bibitem{Asinowski2}
A. Asinowski, C. Banderier, S. Billey, B. Hackl, and S. Linusson. Pop-stack sorting and its image: permutations with overlapping runs. \emph{Acta. Math. Univ. Comenian.}, {\bf 88} (2019), 395--402. 

\bibitem{Barnard}
E. Barnard. The canonical join complex. \emph{Electron. J. Combin.}, {\bf 26} (2019). 

\bibitem{BarnardPop}
E. Barnard, C. Defant, and E. J. Hanson. Pop-stack operators for torsion classes and Cambrian lattices. \emph{Preprint}: arXiv:2312.03959. 

\bibitem{CarrDevadoss}
M. Carr and S. L. Devadoss. Coxeter complexes and graph associahedra. \emph{Topology Appl.}, {\bf 153} (2006), 2155--2168. 


\bibitem{ChoiSun}
Y. Choi and N. Sun. The image of the Pop operator on various lattices. \emph{Adv. Appl. Math.}, {\bf 154} (2024).  

\bibitem{ClaessonPop}
A. Claesson and B. \'A. Gu{\dh}mundsson. Enumerating permutations sortable by $k$ passes through a pop-stack. \emph{Adv. Appl. Math.}, {\bf 108} (2019), 79--96.  

\bibitem{ClaessonPantone}
A. Claesson, B. \'A. Gu{\dh}mundsson, and J. Pantone. Counting pop-stacked permutations in polynomial time. \emph{Experiment. Math.}, (2021).  

\bibitem{DefantCoxeterPop}
C. Defant. Pop-stack-sorting for Coxeter groups. \emph{Comb. Theory}, {\bf 2} (2022).  

\bibitem{DefantMeeting}
C. Defant. Meeting covered elements in $\nu$-Tamari lattices. \emph{Adv.\ Appl.\ Math.}, {\bf 134} (2022). 

\bibitem{DefantSack}
C. Defant and A. Sack. Operahedron lattices. \emph{Preprint}: arXiv:2402.12717. 


\bibitem{Semidistrim}
C. Defant and N. Williams. Semidistrim lattices. \emph{Forum Math. Sigma}, {\bf 11} (2023). 

\bibitem{Galashin}
P. Galashin. $P$-associahedra. \emph{Selecta Math.}, {\bf 30} (2023). 

\bibitem{Hong}
L. Hong. The pop-stack-sorting operator on Tamari lattices. \emph{Adv. Appl. Math.}, {\bf 139} (2022). 

\bibitem{Laplante}
G. Laplante-Anfossi. The diagonal of the operahedra. \emph{Adv. Math.}, {\bf 405} (2022). 

\bibitem{Lichev}
L. Lichev. Lower bound on the running time of pop-stack sorting on a random permutation. arXiv:2212.09316(v1). 

\bibitem{Postnikov} 
A. Postnikov. Permutohedra, associahedra, and beyond. \emph{Int. Math. Res. Not. IMRN}, {\bf 2009} (2009), 1026--
1106.

\bibitem{PudwellSmith}
L. Pudwell and R. Smith. Two-stack-sorting with pop stacks. \emph{Australas. J. Combin.}, {\bf 74} (2019), 179--195. 

\bibitem{Sapounakis}
A. Sapounakis, I. Tasoulas, and P. Tsikouras. On the dominance partial ordering on Dyck paths. \emph{J. Integer Seq.}, {\bf 9} (2006). 

\bibitem{Ungar}
P. Ungar. $2N$ noncollinear points determine at least $2N$ directions. \emph{J. Combin. Theory Ser. A}, {\bf 33} (1982), 343--347. 

\end{thebibliography}
\end{document}